\documentclass[12pt,a4paper]{article}
\usepackage[utf8]{inputenc}
\usepackage{amsmath}
\usepackage{amsfonts}
\usepackage{amssymb}
\usepackage{amsthm}
\usepackage[colorlinks=true,linkcolor=blue]{hyperref}
\hypersetup{
	citecolor={blue}
}
\usepackage[left=2cm,right=2cm,top=2cm,bottom=2cm]{geometry}
\usepackage{xcolor}
\usepackage[all,cmtip]{xy}
\usepackage{tikz}
\usepackage{tikz-cd}
\usetikzlibrary{decorations.markings}
\usepackage{enumerate}
\usepackage{enumitem}
\setlist[enumerate]{label = (\roman*)}
\graphicspath{{bilder/}}
\usepackage[
  backend=biber,
  style=numeric,
  sorting=nyt,
  doi=false,
  url=false,
  isbn=false,
  giveninits=true
  ]{biblatex}
\renewbibmacro{in:}{}
\usepackage{fancyvrb}

\newcommand{\GL}{\mathrm{GL}}

\newcommand{\Th}{\mathrm{Th}}

\newcommand{\CC}{\mathbb{C}}
\newcommand{\ZZ}{\mathbb{Z}}

\newcommand{\PP}{\mathbb{P}}

\newtheorem{thm}{Theorem}[section]
\newtheorem{prop}[thm]{Proposition}
\newtheorem{cor}[thm]{Corollary}
\newtheorem{lem}[thm]{Lemma}

\theoremstyle{definition}
\newtheorem{rem}[thm]{Remark}
\newtheorem{defn}[thm]{Definition}
\newtheorem{ex}[thm]{Example}

\begin{document}

\title{On the Stiefel--Whitney classes of GKM manifolds}
\author{Oliver Goertsches\footnote{Philipps-Universit\"at Marburg, email:
goertsch@mathematik.uni-marburg.de}, Panagiotis Konstantis\footnote{Universität zu Köln,
email: pako@mathematik.uni-koeln.de}, and Leopold
Zoller\footnote{Ludwig-Maximilians-Universit\"at M\"unchen, email: leopold.zoller@mathematik.uni-muenchen.de}}

\maketitle

\begin{abstract} 
We show that under standard assumptions on the isotropy groups of an integer GKM
manifold, the equivariant Stiefel--Whitney classes of the action are determined
by the GKM graph. This is achieved via a GKM-style description of the
equivariant cohomology with coefficients in a finite field $\ZZ_{p}$ even though in this setting the
restriction map to the fixed point set is not necessarily injective. This closes a gap in our
argument why the GKM graph of a $6$-dimensional integer GKM manifold determines
its nonequivariant diffeomorphism type. We introduce combinatorial
Stiefel--Whitney classes of GKM graphs and use them to derive a nontrivial
obstruction to realizability of GKM graphs in dimension $8$ and higher.
\end{abstract}

\section{Introduction}

In GKM theory, named after an influential paper of Goresky--Kottwitz--MacPherson
\cite{MR1489894}, one associates a labelled graph to certain torus actions on smooth manifolds. Concretely, we consider closed orientable manifolds $M$ with vanishing odd degree cohomology on which a compact torus $T$ acts with finitely many fixed points, such that in each fixed points, the isotropy weights are pairwise linearly independent. In this situation, the orbit space of the one-skeleton of the action is homeomorphic to a graph, and we label its edges with the corresponding isotropy weights. The benefit of this GKM graph of the action is that it encodes a multitude of topological properties, both of the action and of the manifold acted on, such as the (equivariant) cohomology ring.

In order for such statements to hold true with integer coefficients one assumes
additionally that for every point outside the one-skeleton of the action, i.e.,
every point $q\in M$ with $\dim T_q<\dim T-1$, the isotropy group $T_q$ is
contained in a proper subtorus of $T$. This condition is encoded in the GKM graph (see Remark \ref{rem:cond1}) and holds e.g.\ if the isotropy groups are connected. The reason for imposing this condition is that it ensures the Chang--Skjelbred Lemma \cite{MR375357} to be valid with integer coefficients \cite{MR2790824}.

In this note we derive, under the very same assumption, a GKM type description of the equivariant cohomology with $\ZZ_p$ coefficients, see Theorem \ref{T: InjectionInZp}. This is insofar remarkable as the standard method to prove
such statements, namely by embedding the equivariant cohomology into the
equivariant cohomology of the fixed point set, is not applicable here. In fact,
the map induced by the inclusion of the fixed point set is, under our
assumptions, not necessarily injective with coefficients $\ZZ_p$. As a
replacement, we consider the natural map from the disjoint union of the
invariant $2$-spheres into $M$, and show that it induces an injection in
equivariant cohomology for arbitrary coefficients, see Lemma
\ref{lem:inclusionofspheres}. It turns out that the $H^*(BT;\ZZ_p)$-algebra $H^\ast_T(M;\ZZ_p)$ is naturally embedded in the sum $H^\ast_T(\Gamma;\ZZ_p)\oplus B^\ast(\Gamma,p)$ (equipped with an appropriate algebra structure). Here, $H^\ast_T(\Gamma;\ZZ_p)$ is equivariant graph cohomology with $\ZZ_p$ coefficients, see Definition \ref{defn:eqgraphcohom}, and $B^\ast(\Gamma,p) =\bigoplus_{e\in E(\Gamma,p)} H^{\ast-2}(BT;\ZZ_p)$, where $E(\Gamma,p)$ is the set of isotropy weights that are divisible by $p$. 

We give a description of the total equivariant Stiefel--Whitney class as an element of $H^\ast_T(\Gamma;\ZZ_2)\oplus B^\ast(\Gamma,2)$ in terms of the GKM graph (Theorem \ref{Thm:SWclass}). As a corollary we prove that under our assumptions on the isotropy groups, the (equivariant)
Stiefel--Whitney classes are encoded
in the GKM graph. This statement was already used in \cite{MR4088417} in the proof of the fact that the diffeomorphism type of a $6$-dimensional simply-connected GKM manifold is determined by its GKM graph. This gap in the argument is filled by the present note, see Remark \ref{rem:GKZ6Korrektur}. As another application we derive a combinatorial criterion of when a GKM manifold admits an (equivariant) spin structure (Theorem \ref{T: Equivariant spin structures}). Our computation of the equivariant Stiefel--Whitney classes of GKM manifolds motivates a definition of combinatorial Stiefel--Whitney classes for abstract GKM graphs. These are certain elements in $H^\ast_T(\Gamma;\ZZ_2)\oplus B^\ast(\Gamma,2)$, see Definition \ref{defn:eqSW}.

We note that a purely algebraic description of $H_T^*(M;\mathbb{Z}_p)$ is in some sense unremarkable as it can be obtained by tensoring $H_T^*(M;\mathbb{Z})$ with $\mathbb{Z}_p$, and our assumptions ensure that the classical GKM description of $H_T^*(M;\mathbb{Z})$ is valid. However this method is not apt to describe phenomena that are intrinsic to finite coefficients. In fact, the discrepancy between $H_T^*(M;\mathbb{Z})\otimes \mathbb{Z}_p$ and our more $\ZZ_p$-intrinsic description has interesting implications: we show by means of an example that in dimension $8$ and higher, the condition that the combinatorial Stiefel--Whitney classes are contained in $H^*_T(M;\ZZ_2)$ provides a nontrivial obstruction to realizability of abstract GKM graphs, see Theorem \ref{thm:nonrealizable}. For $T^2$-actions in dimension $6$, i.e., for $3$-valent graphs with labels in $\mathbb{Z}^2$, this condition turns out to be contained in the known obstructions for realizability, see Proposition \ref{prop:SW3valent}, which explains why it does not appear as a separate condition in the realization results of \cite{2210.01856v1}. \\

\noindent {\bf Acknowledgements.} We are grateful to Michael Wiemeler for pointing out to us the gap in \cite{MR4088417} mentioned above.

\section{GKM actions}

The type of Lie group actions we consider in this paper are the following, named
after \cite{MR1489894}.

\begin{defn}
An action of a compact torus $T=S^1\times \ldots \times S^1$ on an even-dimensional smooth closed orientable manifold $M$ with $H^{odd}(M;\ZZ)=0$ is called \emph{(integer) GKM} if
\begin{enumerate}
\item its fixed point set $M^T = \{p\in M\mid  T\cdot p = \{p\}\}$ is finite and
\item its one-skeleton $M_1:=\{p\in M\mid \dim T\cdot p\leq 1\}$ is a finite union of $T$-invarant $2$-spheres.
\end{enumerate}
The manifold $M$, together with an integer GKM action, will be called an \emph{(integer) GKM manifold}.
\end{defn}
Often, one also considers rational GKM manifolds, for which one instead requires the rational odd cohomology to vanish. In this paper, only the more restrictive integer case will be relevant.

Given a GKM $T$-action on $M$, the orbit space $M_1/T$ is homeomorphic to a graph, which we will denote by $\Gamma$. Its vertex set equals $M^T$, and every (unoriented) edge connecting two vertices $p$ and $q$ represents a $T$-invariant $2$-sphere $S$ containing the fixed points $p$ and $q$. We equip such an edge $e$ with the label $\alpha(e)\in H^2(BT;\ZZ)/\pm 1\cong \ZZ^{\dim T}/\pm 1$ given by the weight of the $T$-module $T_pS$ (which is, as a real representation, isomorphic to $T_qS$).  This labeling turns $\Gamma$ into what is known as an \emph{abstract GKM graph}, a notion that we will recall now.

Let $\Gamma$ be an $n$-valent graph with finite vertex set $V(\Gamma)$ and finite edge set $E(\Gamma)$. We assume that $E(\Gamma)$ does not contain loops, but multiple edges between vertices are allowed. The edges of $\Gamma$ do not come with a fixed orientation, but if we consider on an edge $e\in E(\Gamma)$ an orientation, then it we can speak about its initial vertex $i(e)$ and its terminal vertex $t(e)$. For an oriented edge $e$, we denote by $\bar{e}$ the same edge with opposite orientation. Given a vertex $v$ we denote by $E_v$ the set of all oriented edges emanating from $v$ with starting vertex $v$.
\begin{defn}
A \emph{connection} on $\Gamma$ is a collection of bijections $\nabla_e:E_{i(e)}\to E_{t(e)}$, for any oriented edge $e$, such that
\begin{enumerate}
\item $\nabla_e (e) = \bar{e}$ and
\item $\nabla_{\bar{e}} = \nabla_e^{-1}$.
\end{enumerate}
\end{defn}
\begin{defn}
Let $k\geq 1$. Then an \emph{(abstract) GKM graph} is an $n$-valent graph $\Gamma$, together with a labelling of the edges $\alpha:E(\Gamma)\to \ZZ^k/\pm 1$, called \emph{axial function}, such that
\begin{enumerate}
\item For any $v\in V$ and $e\neq f\in E_v$, the labels $\alpha(e)$ and $\alpha(f)$ are linearly independent.
\item There exists a connection $\nabla$ on $\Gamma$ which is compatible with
	the axial function in the sense that for every oriented edge $e$ and $f\in
	E_{i(e)}$ there exists sign choices for $\alpha(f)$ and $\alpha(\nabla_{e}f)$
	such that
\[
\alpha(\nabla_e f) \equiv \alpha(f) \mod \alpha(e).
\]
\end{enumerate} 
\end{defn}
It was shown in \cite[Proposition 2.3]{MR4363804} and \cite{MR1823050} that the graph associated to a GKM $T$-action is an abstract GKM graph, i.e., that it admits a connection compatible with the labelling. Note however that the compatible connection is not necessarily unique, so that we do not fix it as part of the structure.

Recall that the equivariant cohomology of a $T$-action on a $T$-space $M$ with coefficients in a ring $A$ is defined as the cohomology $H^*_T(M;A):=H^*(M_T;A)$ of the Borel construction $M_T=M\times_T ET$. 

Throughout the paper we will often impose the following assumption on the action:
\begin{equation}\label{eq:standingassumption}
\text{For every }q\in M\setminus M_1\text{, the isotropy group }T_q\text{ is contained in a proper subtorus of }T.
\end{equation}

	The reason for this requirement is the Chang--Skjelbred Lemma for integer
	coefficients, see \cite[Theorem 2.1]{MR2790824}, which is crucial for many of our considerations. We note that the Lemma holds under more general topological assumptions than the ones made below.
\begin{lem}\label{lem:changskjelbred} For a closed smooth $T$-manifold $M$ satisfying \eqref{eq:standingassumption} such that $H^\ast_T(M;\ZZ)$ is a free $H^\ast(BT;\ZZ)$-module, the sequence
\[
0\longrightarrow H^*_T(M;\ZZ)\longrightarrow H^*_T(M^T;\ZZ)\longrightarrow H^{*+1}_T(M_1,M^T;\ZZ)
\]
is exact, where the middle arrow is induced from the inclusion $M^T\to M$, and the right arrow is the boundary morphism in the long exact sequence in equivariant cohomology of the pair $(M_1,M^T)$.
\end{lem}
This implies that the integral equivariant cohomology of the $T$-action is determined by the one-skeleton: the morphism $H^*_T(M;\ZZ)\to H^*_T(M^T;\ZZ)$ is injective, and its image equals the image of the map $H^*_T(M_1;\ZZ)\to H^*_T(M^T;\ZZ)$ induced by the inclusion $M^T\to M_1$. 

If the action is GKM, then
freeness of the equivariant cohomology is implied by the vanishing of the odd
cohomology groups and the one-skeleton $M_1$ is, as a $T$-space, encoded in the
GKM graph, so Condition \eqref{eq:standingassumption} implies via Lemma \ref{lem:changskjelbred} that $H^*_T(M;\ZZ)$ is fully described by the GKM graph. One defines 
\begin{defn}\label{defn:eqgraphcohom}
The \emph{equivariant graph cohomology} of the GKM graph $\Gamma$ with coefficient ring $A$ is
\begin{align*}
H^*_T(\Gamma;A)&:= \{(f_p)\in \bigoplus_{p\in M^T}H^*(BT;A)\mid f_{q}-f_{r} \equiv 0 \mod\alpha(e) \text{ for all edges }e\text{ between }q \text{ and }r\}\\
&\subset H^*_T(M^T;A).
\end{align*}
where we consider $\alpha(e)$ as an element in the image of $H^2(BT;\ZZ) \to H^2(BT;\ZZ)\otimes A = H^2(BT;A)$ modulo sign. On $H^*_T(\Gamma;A)$ we consider the natural structure of $H^*(BT;A)$-algebra, with componentwise multiplication.
\end{defn}
The same definition was given in \cite{MR1823050} for abstract GKM graphs, keeping in mind the identification $H^2(BT;\ZZ)\cong \ZZ^k$.

As observed in \cite[Theorem 7.2]{MR1489894} (for coefficients in the complex numbers; see \cite[Proposition 2.30]{goertsches2022lowdimensional} for the integer case), the Chang--Skjelbred Lemma becomes
\begin{prop}
For an integer GKM action satisfying \eqref{eq:standingassumption}, the inclusion $M^T\to M$ induces an isomorphism
\[
H^*_T(M;\ZZ)\cong H^*_T(\Gamma;\ZZ).
\]
\end{prop}
See Section \ref{Sec:GKMForZp} below, in particular Theorem \ref{T: InjectionInZp}, for a relation between $H^*_T(M;\ZZ_p)$ and $H^*_T(\Gamma;\ZZ_p)$.

\begin{rem}\label{rem:cond1}
As the previous proposition is central to many aspects of this paper, so is
Condition \eqref{eq:standingassumption}. In our setup, this condition is in
fact encoded directly in the GKM graph as described by Proposition
\ref{prop:cond1ingraph} below. This was observed earlier in
\cite[Proposition 3.10]{charton2023monotone} in a symplectic setup.
\end{rem}

\begin{prop}\label{prop:cond1ingraph}
An integer GKM manifold satisfies Condition \eqref{eq:standingassumption} if and only if any two adjacent weights are coprime (i.e.\ they are not both divisible by any $1<n\in \mathbb{Z}$).
\end{prop}

\begin{proof}
If two weights $\alpha,\beta$ adjacent to $q\in M^T$ are divisible by a prime $p$, then the corresponding $4$-dimensional subspace in $T_q M$ is fixed by the maximal $p$-torus $T_p$. This subgroup is not contained in a proper subtorus hence Condition \eqref{eq:standingassumption} is violated in a neighbourhood of $x$.

Now assume that any two adjacent weights are coprime and fix a prime $p$. Then $M_1^{T_p}$ consists of $M^T$ as well as all invariant $2$-spheres whose weight is divisible by $p$. In particular $\dim_{\mathbb{Z}_p}H^*(M_1^{T_p};\mathbb{Z}_p)=|M^T|$. Let $S$ be an invariant $2$-sphere in $M_1^{T_p}$ and let $NS$ denote its normal bundle. Then by assumption none of the weights in the isotropy representations of $N|_{S^T}$ are divisible by $p$. From this we infer that $(NS)^{T_p}$ is just the zero section. An analogous argument for the isolated fixed points shows that $M_1^{T_p}$ is a union of connected components of $M^{T_p}$. But we have
\[\dim_{\mathbb{Z}_p} H^*(M^{T_p};\mathbb{Z}_p)\leq \dim_{\mathbb{Z}_p} H^*(M;\mathbb{Z}_p)=\dim_\mathbb{Q} H^*(M;\mathbb{Q})=|M^T|\]
where the inequality follows from the localization theorem applied to $T_p$ (see
e.g.\ \cite[Thm.\ 3.10.4]{MR1236839}, the first equality follows from
$H^*(M;\mathbb{Z})$ being torsion free, and the second equality is due to the
fact that the spectral sequence of the Serre spectral sequence of the Borel
fibration of the $T$-action collapses (see e.g.\ \cite[Thm.\ 3.10.4]{MR1236839}. Hence it follows that $M^{T_p}$ contains no connected components besides the ones in $M_1^{T_p}$. Since any subgroup of $T$ which is not contained in a proper subtorus contains a maximal $p$-torus for some prime, this proves that Condition \eqref{eq:standingassumption} holds (to see the last claim, note that any closed subgroup of $T$ is isomorphic to $\mathbb{Z}_{n_1}\times\ldots\times \mathbb{Z}_{n_k}\times T^l$ with $n_i>1$ and $n_i|n_{i+1}$ and being contained in a proper subtorus is equivalent to $k+l<\dim T$).
\end{proof}

Let us recall the notion of an \emph{orientable} GKM graph introduced in
\cite{2210.01856v1}. Given an abstract GKM graph $(\Gamma,\alpha)$ with labels in $\mathbb{Z}^k/\pm 1$ we choose an arbitrary compatible connection and lift $\tilde{\alpha}\colon E(\Gamma)\rightarrow \mathbb{Z}^k$. Now for an edge $e\in E(\Gamma)$ and $e\neq f\in E_{i(e)}$ there is a unique $\epsilon_f\in\{\pm1\}$ such that \[\tilde{\alpha}(f)\equiv \epsilon_f\tilde{\alpha}(\nabla_e f)\mod \alpha(e)\]
We set \[\eta(e)=-\prod_{f\in E_{i(e)}\backslash\{e\}} \epsilon_f.\]
\begin{defn}\label{defn:orientablegkmgraph}
We call the abstract GKM graph $(\Gamma,\alpha)$ \emph{orientable} if for every closed edge path $e_1,\ldots,e_l$ in $\Gamma$ one has $\eta(e_1)\cdot\ldots\cdot\eta(e_l)=1$.
\end{defn}
As shown in \cite{2210.01856v1} this property is independent of the choices of
$\nabla$ and $\tilde{\alpha}$. The GKM graph of a GKM manifold is always
orientable, see \cite[Corollary 2.24]{2210.01856v1}.

\section{A GKM description of $H^*_T(M;\ZZ_p)$}\label{Sec:GKMForZp}

The starting point of our description is given by the following

\begin{lem}\label{lem:inclusionofspheres}
Let $M$ be an integer GKM manifold satisfying Condition \eqref{eq:standingassumption} and let $X_M$ denote the disjoint union of all invariant $2$-spheres of $M$. Then for any coefficient ring $A$ the map $H_T^*(M;A)\rightarrow H^*_T(X_M;A)$ is injective.
\end{lem}
 
 \begin{proof}
 Let $S=H^+(BT;A)$. We set $M^S$ to be the collection of all points $q\in M$ for
 which no element of $S$ gets annihilated under the map $H^*(BT;A)\rightarrow
 H^*(BT_q;A)$ where $T_q$ is the stabilizer of $x$. Let $U\subset T$ be a proper
 subtorus. Then the induced map $H^2(BT;A)\rightarrow H^2(BU;A)$ has nontrivial
 kernel as up to isomorphism it is a projection $A^{\dim(T)}\rightarrow
 A^{\dim(U)}$. In particular some element of $S$ gets annihilated by this map.
 Hence, by the assumption on the isotropies, the same holds for
 $H^\ast(BT;A)\rightarrow H^\ast(BT_q;A)$ for any point $q\notin M_1$.
 Consequently $M^S\subset M_1$. As $H_T^\ast(M;A)$ is free over $H^\ast(BT;A)$
 the localization theorem (see e.g.\ \cite[Thm.\ 3.2.6]{MR1236839} implies the injectivity of $H_T^\ast(M;A)\rightarrow H_T^\ast(M^S;A)$ and hence also
 \[H_T^\ast(M;A)\rightarrow H_T^\ast(M_1;A)\]
 is injective. Since $H_T^\ast(M)$ is concentrated in even degrees the claim of the lemma will follow from the fact that
 \[H_T^\ast(M_1;A)\rightarrow H_T^\ast(X_M;A)\]
 is injective in even degrees. To see this thicken the vertices of $M_1$ to starlike trees on which $T$ acts trivially to obtain an equivariantly homotopy equivalent space $Y$. Now cover $Y=V\cup W$ where $V$ is a small neighbourhood of $X_M\subset Y$ and $W$ is the interior of the inserted trees. The corresponding  Mayer-Vietoris sequence then reads
 \[\ldots \rightarrow H_T^\ast(M_1;A)\rightarrow H_T^\ast(X_M;A)\oplus H^\ast_T(W;A)\rightarrow H^\ast_T(V\cap W;A)\rightarrow\ldots\]
 We note that $V\cap W$ is a disjoint union of intervals fixed by $T$. Hence the map $H_T^\ast(W;A)\rightarrow H_T^\ast(V\cap W;A)$ is injective as well as concentrated in even degrees. This proves the desired injectivity of $H_T^{2\ast}(M_1;A)\rightarrow H_T^{2\ast}(X_M;A)$.
 \end{proof}

In order to achieve a combinatorial description of the equivariant cohomology we recall the equivariant cohomology of $2$-spheres.
Let $S^2_\alpha$ be $S^2$, equipped with the $T$-action with weight $\alpha \in
H^{2}(BT;\ZZ)$. We assume that the weight is nontrivial; then the action is (potentially noneffective) integer GKM, and its integer equivariant cohomology is, via restriction to the fixed point set, given as
\begin{equation}\label{eq:S2CS}
H^*_T(S^2_\alpha;\ZZ) = \{(f,g)\in H^*(BT;\ZZ)^2\mid f-g\equiv 0 \mod \alpha\}.
\end{equation}
It is a free $H^*(BT;\ZZ)$-module of rank $2$; one choice of generators are $1$ and $(\alpha,0)$ (or $(0,\alpha)$).

Let $p$ be an arbitrary prime. By the universal coefficient theorem, 
\[
H^*_T(S^2_\alpha;\ZZ_p) = H^*_T(S^2_\alpha;\ZZ)\otimes \ZZ_p,
\]
which is a free $H^*(BT;\ZZ_p)$-module of rank $2$.
\begin{lem}\label{L:EquivariantCohmologySalpha}
We denote $\xi:=(\alpha,0)$ as an element of $H^2_T(S^2_\alpha;\ZZ)$ under the identification \eqref{eq:S2CS}. With the same letter we denote its reduction to $\ZZ_p$ coefficients.
\begin{enumerate}
\item As an $H^*(BT; \ZZ)$-algebra,
\[
H^*_T(S^2_\alpha;\ZZ) = H^*(BT;\ZZ)[\xi]/(\xi^2-\alpha \xi).
\]
\item As an $H^*(BT;\ZZ_p)$-algebra,
\[
H^*_T(S^2_\alpha;\ZZ_p) = H^*(BT;\ZZ_p)[\xi]/(\xi^2-\alpha\xi).
\]
In particular, if $\alpha$ is divisible by $p$, then
\[
H^*_T(S^2_\alpha;\ZZ_p) = H^*(BT;\ZZ_p)[\xi]/(\xi^2).
\]
\item If $\alpha$ is not divisible by $p$, then the inclusion of the fixed point set induces an injection
\[
H^*_T(S^2_\alpha;\ZZ_p)\to H^*(BT;\ZZ_p)\oplus H^*(BT;\ZZ_p).
\]
with image $\{(f,g)\mid f-g \equiv 0 \mod r(\alpha)\}$, where $r\colon H^*(BT;\mathbb{Z})\rightarrow H^*(BT;\mathbb{Z}_p)$ reduces coefficients.
\item If $\alpha$ is divisible by $p$, then the map 
\begin{align*}
\{(f,g)\in H^*(BT;\ZZ)^2\mid &\,\, f-g\equiv 0 \mod \alpha\}\\
&=H^*_T(S^2_\alpha;\ZZ)\to H^*_T(S^2_\alpha;\ZZ_p) = H^*(BT;\ZZ_p)[\xi]/(\xi^2)
\end{align*} given by reducing coefficients from $\ZZ$ to $\ZZ_p$ is the map
\[
(f,g)\mapsto r(g) + r\left(\frac{f-g}{\alpha}\right)\xi.
\]
The inclusion of the fixed point set
\[
H^*_T(S^2_\alpha;\ZZ_p)\to H^*(BT;\ZZ_p)\oplus H^*(BT;\ZZ_p).
\]
sends $1\in H^*_T(S^2_\alpha;\ZZ_p)=H^*(BT;\ZZ_p)[\xi]/(\xi^2)$ to $(1,1)$ and has kernel generated by $\xi$.
\end{enumerate}
\end{lem}
\begin{proof}
For part (i), note that $1$, together with $\xi=(\alpha,0)$, consitutes an integer basis of $H^*_T(S^2_\alpha;\ZZ)$, and that $\xi$ satisfies $\xi^2=\alpha\xi$. Part (ii) follows immediately by reducing coefficients to $\ZZ_p$. 
Part $(iii)$ follows from the integral GKM description $H^*_T(S^2_\alpha;\ZZ)\cong \{(f,g)\in H^*(BT;\ZZ)^2\mid f-g\equiv 0 \mod \alpha\}$ by reducing coefficients to $\mathbb{Z}_p$. Indeed, by naturality the restriction $H^*_T(S^2_\alpha;\ZZ_p)\to H^*(BT;\ZZ_p)\oplus H^*(BT;\ZZ_p)$ takes values in the desired subalgebra and for any $(f,g)$ in said subalgebra we find lifts of $f,g$ to $H^*(BT;\mathbb{Z})$ that agree modulo $\alpha$. To see injectivity, let $x\in H_T^*(S^2_\alpha;\mathbb{Z}_p)$ lie in the kernel and choose a lift $y\in H_T^*(S^2_\alpha;\mathbb{Z})$. Then $y$ maps to $(f,g)\in H^*(BT;\ZZ)\oplus H^*(BT;\ZZ)$ where $f,g\equiv 0\mod p$. But since $\alpha$, $p$ are coprime also $\frac{f}{p}\equiv \frac{g}{p}\mod \alpha$ and thus $y$ is divisible by $p$ in $H_T^*(M;\mathbb{Z})$. Hence $x=0$.
The first statement in part (iv) follows from reducing coefficients in the expression $(f,g) = (g,g) + (f-g,0)$, and the second one follows directly from our choice of generator $\xi$.
\end{proof}

\begin{rem}\label{rem:signs}
In case $p=2$ and $\alpha$ is divisible by $2$, the generator $\xi$ is (in
$\ZZ_2$ coefficients) uniquely determined by the condition $\xi^2=0$. In fact,
an element of the form $\xi+f$ squares to $(\xi+f)^2 = \xi^2+2f\xi + f^2 = f^2$
which does not vanish if $f\neq 0$. Note that in this case, choosing as second
generator the element $(0,\alpha)$ would give the same element with $\ZZ_2$
coefficients, which is not true for $p>2$. Hence in the latter case the specific
maps of Lemma \ref{L:EquivariantCohmologySalpha} are not canonical.
\end{rem}

In the following, we wish to derive a more explicit description of the
equivariant cohomology $H^*_T(M;\ZZ_p)$ in terms of the GKM graph $\Gamma$. As
in the computation of the equivariant cohomology of the $2$-spheres in Lemma
\ref{L:EquivariantCohmologySalpha} we distinguished one of the two fixed points
via our choice of generator $\xi$, we need to choose an auxiliary orientation of
each edge $e$ of the graph, so that its initial vertex $i(e)$ and its terminal
vertex $t(e)$ are well-defined. This orientation does not need to satisfy any
additional assumptions. Moreover we fix a sign for $\alpha(e)$ for each edge
$e$.

Let $X_{M}$ denote the disjoint union of the invariant $2$-spheres in $M$ as in Lemma \ref{lem:inclusionofspheres}.
Then clearly
\[
	H_{T}^{\ast}(X_{M};\ZZ_{p}) = \bigoplus_{e} H^{\ast}_{T}(S_{\alpha(e)}^{2};
	\ZZ_{p})
\]
where the sum runs over all edges $e$ in the GKM graph of $M$. Using Lemma
\ref{L:EquivariantCohmologySalpha} the sum can be embedded into
\begin{equation}\label{eq:X_MEmbedded}
	\bigoplus_{e \in E(\Gamma)} (H^{\ast}(BT;\ZZ_{p}))^{2} \oplus
	\bigoplus_{e \in E(\Gamma,p)}  H^{\ast}(BT;\ZZ_{p}) \cdot \xi_{e},
\end{equation}
where $E(\Gamma)$ denotes the edges of $\Gamma$,  $E(\Gamma,p) = \{e \in E(\Gamma) \mid 
\alpha(e)\equiv 0\mod p\}$ and $\xi_{e}$ is the $\mod p$ reduction of the element $(\alpha(e),0)\in H^*_T(S^2_{\alpha(e)};\ZZ)$ (where the first entry corresponds to $i(e)$, the second to $t(e)$), so that
\[
	H^{\ast}_{T}(S_{\alpha(e)}^{2}; \ZZ_{p}) =
	H^{\ast}(BT;\ZZ_{p})[\xi_{e}]/(\xi_{e}^{2} -
	\alpha(e)\xi_{e}).
\]
Recall the definition of equivariant graph cohomology, Definition \ref{defn:eqgraphcohom}, and consider 
\[
	H^{\ast}_{T}(\Gamma;\ZZ_{p}) \oplus B^\ast(\Gamma,p) 
\]
with $B^\ast(\Gamma,p)  = \bigoplus_{e \in E(\Gamma,p)}H^{\ast-2}(BT;\ZZ_{p})$, as 
an $H^{\ast}(BT;\ZZ_{p})$-algebra, where multiplication is defined as follows:
For $f= (f_{q}), f'=(f'_{q}) \in H^{\ast}_{T}(\Gamma;\ZZ_{p})$ and
$g=(g_{e}), g'=(g'_{e}) \in B^\ast(\Gamma,p)$ define
\[
	(f,g)(f',g') = (ff', fg' + f'g) 
\]
where $((ff')_{q}) = (f_{q}f'_{q})$ and $fg' = (f_{i(e)}g'_{e})$ as well as $f'g
= (f'_{i(e)}g_{e})$ for $e \in E(\Gamma,p)$. Note that for any such edge $e$ we have $f_{i(e)} =
f_{t(e)}$. Finally, there is an embedding
$\eta \colon H_{T}^{\ast}(\Gamma;\ZZ_{p}) \oplus B^\ast(\Gamma,p) \to H^{\ast}_{T}(X_{M}; \ZZ_{p})$ defined by
\[
	((f_{q}), (g_{e}))	\mapsto \left((f_{i(e)} , f_{t(e)})_{e \in E(\Gamma)},
	(g_{e} \xi_{e})_{e \in E(\Gamma,p)}\right).
\]
 We remark that the above map is well-defined and a homomorphism of
 $H^*_T(BT;\ZZ_p)$-algebras due to Lemma \ref{L:EquivariantCohmologySalpha} (ii)
 and (iii). 

\begin{thm}\label{T: InjectionInZp}
The map $i^{\ast} \colon H_{T}^{\ast}(M; \ZZ_{p})
\to H_{T}^{\ast}(X_{M};\ZZ_{p})$ induced by the map $i \colon X_{M} \to M$ factorizes
through a $H^*(BT;\ZZ_p)$-algebra morphism $\Phi \colon H_{T}^{\ast}(M;\ZZ_{p}) \to
H_{T}^{\ast}(\Gamma;\ZZ_{p}) \oplus B^\ast(\Gamma,p)$.  Moreover the following diagram commutes
\[
\begin{tikzcd}
	H_{T}^{\ast}(M;\ZZ) \arrow{d}{} \arrow{r}{}
	& H_{T}^{\ast}(\Gamma;\ZZ) \arrow{d}{\Psi} &  \\
	H^{\ast}_{T}(M; \ZZ_{p})\arrow{r}{\Phi} \arrow[bend right=30]{rr}{i^{\ast}}
	& H^{\ast}_{T}(\Gamma; \ZZ_{p}) \oplus B^\ast(\Gamma,p)  \arrow{r}{\eta} & H_{T}^{\ast}(X_{M};
	\ZZ_{p})
\end{tikzcd}
\]
where the vertical map on the left is the (surjective) reduction modulo p,
whereas the vertical map on the right is defined as
\[
	\Psi \colon H^{\ast}_{T}(\Gamma,\ZZ) \to H^{\ast}_{T}(\Gamma, \ZZ_{p}) \oplus
	B^\ast(\Gamma,p), \quad 
	\Psi((f_{q})) = \left((r(f_{q})), \left(r\left(\frac{f_{i(e)} -
	f_{t(e)}}{\alpha(e)}\right)\right)_{e \in E(\Gamma,p)} \right).
\]
If Condition \eqref{eq:standingassumption} is valid, then $\Phi$ is injective.
\end{thm}

Recall that above we chose signs and orientations of the edges. These enter in
the definition of $\Psi$ in the case $p \neq 2$, see Remark \ref{rem:signs}.

\begin{proof}
The inclusion of the fixed point set $M^{T} \to M$ factors through $X_{M}$, thus
the induced map factorizes as $H_{T}^{\ast}(M; \ZZ_{p}) \stackrel{i^{\ast}}{\longrightarrow}
H_{T}^{\ast}(X_{M};\ZZ_{p}) \stackrel{j}{\longrightarrow}
H^{\ast}_{T}(M^{T};\ZZ_{p})$.
Now consider the commutative diagram
\[
	\begin{tikzcd}
	 & H_{T}^{\ast}(\Gamma; \ZZ_{p}) \oplus B^\ast(\Gamma,p) \arrow{d}{\eta}\arrow{rd}{\pi} & \\
		H_{T}^{\ast}(M;\ZZ_{p}) \arrow{r}{i^{\ast}} & H_{T}^{\ast}(X_{M}; \ZZ_{p})
		\arrow{r}{j} &
		H_{T}^{\ast}(M^{T};\ZZ_{p}),\\
	\end{tikzcd} 
\]
where $\eta$ is the embedding from above. From the definitions of all involved
maps, the map $\pi$ is just the projection to the first component. For $x \in
H^{\ast}_{T}(M; \ZZ_{p})$ we have that
\[
	\omega := i^{\ast}(x) - \eta(j \circ i^{\ast}(x), 0) \in \ker j.
\]
From Lemma \ref{L:EquivariantCohmologySalpha} (iii) it follows that $\omega$
lies in $\bigoplus_{e \in E(\Gamma,p)}H_{T}^{\ast}(S^{2}_{\alpha(e)};\ZZ_{p})$ and by
Lemma \ref{L:EquivariantCohmologySalpha} (iv) $\omega$ restricted to each
$H_{T}^{\ast}(S^{2}_{\alpha(e)}; \ZZ_{p})$ ($e \in E(\Gamma,p)$) must be a multiple of
$\xi_{e}$, say $g_{e}\xi_{e}$, for $g_{e} \in H^{\ast-2}(BT;\ZZ_{p})$. Therefore $\omega
= \eta(0, (g_{e})_{e \in E(\Gamma,p)})$ and defining
\[
	\Phi(x) := (j \circ i^{\ast}(x), (g_{e})_{e \in E(\Gamma,p)}) 
\]
proves the first statement of the theorem. We see that from Lemma
\ref{L:EquivariantCohmologySalpha} (iv) and the definition of $\Psi$ that the
square in the diagram of the theorem commutes. Condition \eqref{eq:standingassumption} implies that $i^{\ast}$ is injective by Lemma \ref{lem:inclusionofspheres}, hence in this case also $\Phi$ is injective.
\end{proof}

\begin{rem}
We regard this statement as a GKM description of the equivariant cohomology of $M$ with coefficients $\ZZ_p$  because it describes $H^*_T(M;\ZZ_p)$ purely in terms of the GKM graph $\Gamma$. In fact, the theorem says that it is isomorphic, as an $H^*(BT;\ZZ_p)$-algebra, to the image of the map $\Psi:H^*_T(\Gamma;\ZZ)\to H^*_T(\Gamma;\ZZ_p)\oplus B^\ast(\Gamma,p)$. 

Note that the map $\Psi$ is not necessarily surjective, hence
$\Phi:H^*_T(M;\ZZ_p)\to H^*_T(\Gamma;\ZZ_p)\oplus B^\ast(\Gamma,p)$ not
necessarily an isomorphism. In fact even in case all weights are primitive (and
hence there is no $B^\ast(\Gamma,p)$ summand) the map
$H^*_T(M;\mathbb{Z}_p)\rightarrow H^*_T(\Gamma;\mathbb{Z}_p)$ is not necessarily
surjective. This happens e.g.\ for the $T^2$-action on $S^2\times S^2\times S^2$
with weights $(1,0),(0,1),(1,p)$: a computation shows that in this case
$\dim_{\mathbb{Z}_p} H^2_T(\Gamma;\mathbb{Z}_p)=6$ while
$H^2_T(\Gamma;\mathbb{Z})$ is a free group of rank $5$. 
\end{rem}

\begin{rem}
Let $T_p\subset T$ be the maximal $p$-torus. The injection $\Phi$ in Theorem \ref{T: InjectionInZp} is very much related to the composition
\[H_T^*(M;\mathbb{Z}_p)\rightarrow H_{T_p}^*(M;\mathbb{Z}_p)\rightarrow H_{T_p}^*(M^{T_p};\mathbb{Z}_p)\]
 where the first map is an isomorphism onto the even degree part of the middle algebra and the second map is an injection due to the localization theorem and the fact that the Serre spectral sequence of the Borel fibration of the $T_p$-action collapses (this latter fact is inherited from the analogous property of the $T$-action). From the assumptions on $M$ it follows that $M^{T_p}$ consists exactly of the spheres with weight divisible by $p$ as well as the remaining $T$-fixed points. Hence one naturally has $H^*_T(\Gamma;\mathbb{Z}_p)\oplus B^\ast(\Gamma,p)\subset H_{T_p}^*(M^{T_p};\mathbb{Z}_p)$ and the map $H_{T}^*(M;\mathbb{Z}_p)\rightarrow H_{T_p}^*(M^{T_p};\mathbb{Z}_p)$ factors through $\Phi$.

The image of $\Phi$ agrees with the image of
$H^*_{T_p}(M_{1,p};\mathbb{Z}_p)\rightarrow H_{T_p}^*(M^{T_p};\mathbb{Z}_p)$ by
the Chang-Skjelbred Lemma for finite tori (see \cite{MR4223067},
or \cite[Theorem 4.1]{MR2060474} for the case $p=2$) where $M_{1,p}=\{x\in M~|~
p\geq |T_p\cdot x|\}$. However the space $M_{1,p}$ will be larger than $M_1$ in
case the weights are not pairwise linearly independent when reduced to $\mathbb{Z}_p$
coefficients. In particular $M_{1,p}$ is not directly encoded in the GKM graph
(although much of its combinatorics are). This is the reason why our
combinatorial description of the image of $\Phi$ is not intrinsic to
$\mathbb{Z}_p$-coefficients but rather via reduction from the description
obtained by the integral Chang-Skjelbred Lemma.

Under additional conditions on the weights, combinatorial descriptions which are
closer to the classical GKM description can be derived. For example, one may consider so-called $\mod 2$ GKM manifolds \cite{MR2060474}, which are GKM manifolds such that the weights at any fixed point reduced modulo $2$ are non-zero and distinct. This condition is rather restrictive though; for instance, if $\dim T = 2$, then it forces $M$ to be at most $6$-dimensional.

\end{rem}

\begin{rem}\label{rem:naturality}
Let $M$, $N$ be two integer GKM manifolds satisfying Condition
\eqref{eq:standingassumption}. An isomorphism $\Phi\colon
(\Gamma_M,\alpha_M)\rightarrow (\Gamma_N,\alpha_N)$ of GKM graphs is an
isomorphism $\varphi\colon \Gamma_M\rightarrow \Gamma_N$ of graphs together with
an automorphism $\psi\colon T\rightarrow T$ intertwining the labels i.e.\
$\alpha_N(\varphi(e))=\psi^*(\alpha_{M}(e))$ where $\psi^*$ denotes the corresponding automorphism of $H^*(BT;\mathbb{Z})$. Then this induces an isomorphism
\[H_T^*(M;\mathbb{Z})\cong H^*_T(\Gamma_M;\mathbb{Z})\rightarrow H^*_T(\Gamma_N;\mathbb{Z})\cong H_T^*(N;\mathbb{Z})\]
by applying $\psi^*$ to all $H^*(BT;\mathbb{Z})$ components and identifying vertices according to $\varphi$. Similarly $\Phi$ induces an isomorphism
\[H_T^*(\Gamma_M;\mathbb{Z}_p)\oplus B^\ast(\Gamma,p)\rightarrow H_T^*(\Gamma_N;\mathbb{Z}_p)\oplus B^\ast(\Gamma,p).\]

The description from Theorem \ref{T: InjectionInZp} is natural in the sense that the diagram
\[\xymatrix{
H^*_T(M;\mathbb{Z}_p)\ar[r]^{\cong}& H_T^*(M;\mathbb{Z})\otimes \mathbb{Z}_p\ar[r]\ar[d] & H^*_T(N;\mathbb{Z})\otimes \mathbb{Z}_p\ar[r]^{\cong}\ar[d]& H_T^*(N;\mathbb{Z}_p)\\
&H_T^*(\Gamma_M;\mathbb{Z}_p)\oplus B^\ast(\Gamma,p)\ar[r]&H_T^*(\Gamma_N;\mathbb{Z}_p)\oplus B^\ast(\Gamma,p)&
}\]
commutes, provided the choice of signs and orientation of edges in $E(\Gamma, p)$ for the
construction of the vertical maps are compatible.
 \end{rem}

As an aside, we note that the passage to $X_M$ instead of the fixed point set in order to arrive at a GKM description of $H^\ast_T(M;\ZZ_p)$ was necessary:

\begin{prop}\label{prop:injectiveiff} Given Condition \eqref{eq:standingassumption}, the map $H^*_T(M;\ZZ_p)\to H^*_T(M^T;\ZZ_p)$ is injective if and only if none of the weights of the isotropy representations in the fixed points is divisible by $p$.
\end{prop}
\begin{proof}
If some weight $\alpha$ is divisible by $p$, then consider the equivariant Thom
class (cf.\ \cite{MR1823050}) of any of the two fixed points in the corresponding sphere, with integer coefficients. This is, in the description $H^*_T(M;\ZZ)\subset \bigoplus H^*(BT;\ZZ)$, localized at that fixed point, and equal to the product of the weights at that point. Note that no other weight at this fixed point is divisible by $p$, by Condition  \eqref{eq:standingassumption}. Hence, the Thom class is not divisible by $p$ in $H^*_T(M;\ZZ)$. It thus survives to $H^*_T(M;\ZZ_p)$; on the other hand, upon restriction to the fixed point set, it vanishes with $\ZZ_p$-coefficients. Hence, the map induced by the inclusion $M^T\to M$ is not injective with $\ZZ_p$-coefficients.

Conversely, if none of the weights are divisible by $p$,  choose an element $\omega$ in the kernel of the map $H^*_T(M;\ZZ_p)\to H^*_T(M^T;\ZZ_p)$. If it is nonzero, then by Lemma \ref{lem:inclusionofspheres}, there is an invariant $2$-sphere to which $\omega$ restricts nontrivially. But then, by Lemma \ref{L:EquivariantCohmologySalpha} (iii), it restricts nontrivially to the fixed point set, which is a contradiction.\end{proof}

\begin{rem}\label{Rem:ProjectiveBorelModel}
\begin{enumerate}[label=(\alph*)]
	\item Let $\CC_{\alpha}$ denote $\CC$ together with the action of $T$ by
		$\alpha$ and let $T$ act on $\CC_{\alpha} \oplus \CC$ by $\alpha$ on the
		first factor and trivially on the second. Then $S_{\alpha}^{2}$ can be
		identified with the equivariant projectivization $\CC\PP^{1}_{\alpha}$ of
		$\CC_{\alpha} \oplus \CC$. The tautological bundle
		$\mathcal{O}(-1)$ is a subbundle of the equivariant trivial vector bundle
		$\CC\PP^{1} \times (\CC_{\alpha} \oplus \CC)$ given by
		\[
			\left\{([z_{1}:z_{2}],(w_{1},w_{2})) \in \CC\PP^{1} \times (\CC_{\alpha}
				\oplus \CC) :
			(w_{1}, w_{2}) \in [z_{1}:z_{2}] \right\}
		\]
   which has an induced $T$-action
   \[
		 t \cdot ([z_{1}, z_{2}], (w_{1}, w_{2})) = ([\alpha(t)z_{1}, z_{2}],
		 (\alpha(t)w_{1}, w_{2})) 
   \]
	 such that the projection to $\CC\PP^{1}_{\alpha}$ is equivariant.
 \item 
		The Borel model $(\CC\PP^{1}_{\alpha})_{T}$ can be viewed as a
		projectivisation of an equivariant complex rank $2$ vector bundle $E$ over $BT$. To
		see this, let $L_{\alpha} := ET \times_{T} \CC_{\alpha}$ and  
		\[
			E = ET \times_{T} (\CC_{\alpha}\oplus \CC) 
		\]
		where $T$ acts as in (a). 
		We obtain $\PP(E) = (\CC\PP^{1}_{\alpha})_{T}$. Moreover note
		that we identify $\alpha \in H^{2}(BT;\ZZ)$ with the first Chern class
		$c_{1}(L_{\alpha})$. Thus we have $c_{1}(E) = \alpha$ and $c_{2}(E) =0$. Let
		$\mathcal{O}_{\PP(E)}(-1)$ denote the tautological bundle over $\PP(E)$,
		i.e. $\mathcal{O}_{\PP(E)}(-1)$ restricted to
		every fiber of $\PP(E) \to BT$ is the tautological line bundle over
		$\CC\PP^{1}$. In other words we have 
		\[
			\mathcal{O}_{\PP(E)}(-1)  = ET \times_{T} \mathcal{O}(-1).
		\]
		Set $\xi_{\PP} := c_{1}(\mathcal{O}_{\PP(E)}(-1)) \in
		H^{2}_{T}(\CC\PP^{1}_{\alpha};\ZZ)$. Note that $\xi_{\PP}$ is the equivariant first
		Chern class of $\mathcal{O}(-1)$ viewed as an equivariant bundle over
		$\CC\PP^{1}_{\alpha}$, cf.\ (a). If $H_{T}^{\ast}(\CC\PP^{1}_{\alpha}; \ZZ)$
		is embedded in $H^{\ast}(BT;\ZZ)^{2}$ as in \eqref{eq:S2CS} we have 
		\[
			\xi_{\PP} = (\beta_{1}, \beta_{2}) 
		\]
		where $\beta_{i}$ are the weights of the $T$-action of the fibers of
		$\mathcal{O}(-1)$ over the fixed points of $\CC\PP^{1}_{\alpha}$. For
		$\beta_{1}$ we consider the point $[1:0]$ and the action on the fiber is
		given by $t \cdot (\lambda, 0) = (\alpha(t)\lambda, 0)$ for $\lambda \in
		\CC$ (see (a)). Thus for $\beta_{2}$ over $[0:1]$ we have $t \cdot
		(0,\lambda) = (0, \lambda)$ for $\lambda \in \CC$. Hence we obtain
		$\xi_{\PP} = (\alpha,0)$ and consequently in the description of
		$H^{\ast}_{T}(\CC\PP^{1}_{\alpha};\ZZ)$ of Lemma
		\ref{L:EquivariantCohmologySalpha} (i) we have $\xi_{\PP} = \xi$. 	
	\item Let $\mathcal{O}_{\PP(E)}(1)$ be the dual bundle of
		$\mathcal{O}_{\PP(E)}(-1)$. Clearly $\mathcal{O}_{\PP(E)}(1)$ restricted to
		each fiber of $\PP(E)$ is the hyperplane bundle $\mathcal{O}(1)$ of
		$\CC\PP^{1}$. Set $\xi^{\ast} := c_{1}(\mathcal{O}_{\PP(E)}(1))$, thus
		$\xi^{\ast} = -\xi_{\PP}$. As above one has $\xi^{\ast} = (-\alpha,0)$
		as an element of $H^{\ast}(BT; \ZZ)^{2}$.

		Applying the Leray-Hirsch theorem one obtains the cohomology
		of $H^{\ast}(\PP(E);\ZZ)$ as a $H^{\ast}(BT;\ZZ)$-module, also known as the Chow
		ring
		\begin{align*}
			H^{\ast}(\PP(E)) &= H^{\ast}(BT; \ZZ)[\xi^{\ast}]/((\xi^{\ast})^{2} +
			c_{1}(E)\xi^{\ast} + c_{2}(E)) \\
											 &= H^{\ast}(BT; \ZZ)[\xi^{\ast}]/((\xi^{\ast})^{2} +
			\alpha\xi^{\ast})
		\end{align*}
	  Thus in terms of $\xi_{\PP}$	we obtain the same description of the
		cohomology of the Borel model of $(\CC\PP^{1}_{\alpha})_{T}$ as in Lemma \ref{L:EquivariantCohmologySalpha}
		\[
			H^{\ast}(\PP(E)) = H^{\ast}(BT; \ZZ)[\xi_{\PP}]/(\xi_{\PP}^{2} - \alpha
			\xi_{\PP}).
		\]
\end{enumerate}
\end{rem}

 \section{Equivariant Stiefel--Whitney classes}

We denote by $r$ the map that reduces the coefficients of an integer polynomial to $\ZZ_2$. This map is well-defined on polynomials that are given only modulo sign.

\begin{lem}\label{lem:defSW}
Let $(\Gamma,\alpha)$ be a GKM graph. For any $v\in V(\Gamma)$ we set \[f_v=\prod_{e\in E_v}(1+r(\alpha(e)))\in H^*(BT;\mathbb{Z}_2).\] For an edge $e\in E(\Gamma,2)$ we choose a compatible connection $\nabla_e\colon E_{i(e)}\rightarrow E_{t(e)}$ as well as lifts $\tilde{\alpha}(-)$ of labels in $E_{i(e)}$ and $E_{t(e)}$ to $H^*(BT;\mathbb{Z})$ such that $\tilde{\alpha}(l)\equiv \tilde{\alpha}(\nabla_e l)\mod\alpha(e)$ for all $l\in E_{i(e)}$ and set
\[f_e=r\left(\frac{\prod_{l\in E_{i(e)}}(1+\tilde{\alpha}(l))-\prod_{l\in E_{t(e)}} (1+\tilde{\alpha}(l))}{\tilde\alpha(e)}\right)\in H^*(BT;\mathbb{Z}_2).\]
Then this defines an element of $H^*_T(\Gamma;\mathbb{Z}_2)\oplus B^\ast(\Gamma,2)$ which is independent of the choices made in the construction.
\end{lem}
\begin{proof}
The collection of the $f_v$, $v\in V(\Gamma)$ defines an element of $H^*_T(\Gamma;\mathbb{Z}_2)$ due to the existence of a compatible connection. The sign choices are arbitrary as the congruences between the $f_v$ are only checked modulo $2$.

Fixing an edge $e$, and $\nabla, \tilde{\alpha}(-)$ as in the construction of $f_e$, we check first that \[x=\prod_{l\in E_{i(e)}}(1+\tilde{\alpha}(l))-\prod_{l\in E_{t(e)}} (1+\tilde{\alpha}(l))\] is in fact divisible by $\tilde\alpha(e)$ due to the congruences $\tilde{\alpha}(l)\equiv \tilde{\alpha}(\nabla_e l)\mod\alpha(e)$. The element $x$ is determined by two choices: the connection $\nabla_e\colon E_{i(e)}\rightarrow E_{t(e)}$ as well as the subsequent choice of signs for the labels. Inverting the sign of $\tilde\alpha(l)$ for some $l\in E_{i(e)}$ forces a sign change in the corresponding label $\tilde{\alpha}(\nabla_e l)$ and consequently the value $x$ in the construction will differ by \[2\tilde\alpha(l)\prod_{l'\neq l} (1+\tilde{\alpha}(l'))-2\tilde\alpha(\nabla_e l)\prod_{l'\neq l} (1+\tilde{\alpha}(\nabla_e l')).\]
This difference is divisible by $2\tilde{\alpha}(e)$ hence $f_e$ does not depend on the sign choice.

If for some $l,l'\in E_{i(e)}\backslash \{e\}$ one has $\tilde{\alpha}(l)\equiv\pm\tilde{\alpha}(l') \mod \alpha(e)$ then the connection $\nabla$ can be modified to a connection $\nabla'$ by swapping the images of $l,l'$. Any other connection arises from $\nabla$ by operations of this type so it suffices to show that $\nabla'$ admits a choice of signs $\tilde{\alpha}'(-)$ such that the resulting $f_e$ is the same. If $\tilde{\alpha}(l)\equiv \tilde{\alpha}(l') \mod \alpha(e)$ then $\tilde{\alpha}=\tilde{\alpha}'$ works. Otherwise we have to modify $\tilde{\alpha}$ by setting $\tilde{\alpha}'(\nabla_e l)=-\tilde{\alpha}(\nabla_e(l))$ and $\tilde{\alpha}'(\nabla_e l')=-\tilde{\alpha}(\nabla_e l')$. With this modified sign choice the construction of $x$ changes by a multiple of $2(\tilde\alpha(\nabla_e l)+\tilde\alpha(\nabla_e l'))$. However since by assumption $\tilde{\alpha}(\nabla_e l)\equiv\tilde\alpha(l)\equiv -\tilde\alpha(l')\equiv\tilde\alpha(\nabla_e l')\mod\alpha(e)$ this difference is divisible by $2\tilde{\alpha}(e)$ and hence yields the same value for $f_e$.
\end{proof}

\begin{defn}\label{defn:eqSW}
We call the element $f\in H^*_T(\Gamma;\mathbb{Z}_2)\oplus B^\ast(\Gamma,2)$ the \emph{total equivariant Stiefel-Whitney class of $(\Gamma,\alpha)$}.
\end{defn}

\begin{lem}\label{lem:uniqueness} Consider a $T$-action on $S^2$ via a weight $\alpha$ and let $E\rightarrow S^2$ be an orientable real $T$-vector bundle such that the weights over the fixed points are linearly independent from $\alpha$.
Then $E$ is determined up to isomorphism by the weights (up to sign) over the fixed points.
\end{lem} 
 
\begin{proof}
Denote by $N,S$ the fixed points of $S^2$, and by $S^2_N$, $S^2_S$ the corresponding hemispheres. Note that each hemisphere deformation retracts equivariantly to the corresponding fixed point so we obtain isomorphisms $E|_{S^2_N}\cong D^2\times E_N$ where $T$ acts diagonally and analogously for $S$. Having fixed these isomorphisms we identify $E$ as the gluing of $D^2\times E_N$ and $D^2\times E_S$ along an isomorphism \[\varphi\colon S^1\times E_N\rightarrow  S^1\times E_S\] which covers the identity on $S^1$.
We argue that the space of such isomorphisms is connected and hence the isomorphism class resulting from the gluing does not depend on $\varphi$. Consequently the isomorphism class as a real $T$-vector bundle will be determined by the isomorphism classes of the real $T$-representations $E_N$ and $E_S$ and hence by the weights up to sign.

Set $U=\ker(\alpha)$ and let $V$ be the $U$-representation given by the restriction of $E_N$ which gets identified via $\varphi|_{\{1\}\times E_N}$ with the restriction of $E_S$ to $U$. We identify domain and target of $\varphi$ with $T\times_U V$ via the map $[t,v]\mapsto (t\cdot 1, tv)$. Thus using these identifications any isomorphism $\sigma$ like $\varphi$ can be considered a $T$-equivariant automorphism of $T\times_U V$ covering the identity of $T/U$. We conclude the proof by arguing that the space of these automorphisms is connected.

Such an automorphism $\sigma$ induces a unique $A\in \GL(V)$ determined by
$\sigma([1,v])=[1,Av]$ which commutes with the $U$-action. Conversely such an
$A$ defines an automorphism of $T\times_U V$ by setting $\sigma([t,v])=[t,Av]$.
Hence it suffices to prove that the image of $U$ in $\GL(V)$ has connected
centralizer. We note that $U\cong U_0\times G$ where $U_0$ is a torus and $G$ is
a potentially trivial finite cyclic group. The $U_0$-representation $V$
decomposes into a sum of irreducible real representations $V\cong\bigoplus
\mathbb{C}_{\beta_i}^{k_i}$ where the $\beta_i$ are weights $U_0\rightarrow S^1$
with $\beta_i\neq\pm \beta_j$ for $i\neq j$ (note that the real
$U_0$-representations defined by $\beta_i$ and $-\beta_i$ are isomorphic).
Furthermore none of the $\beta_i$ are trivial as by assumption the weights are
linearly independent from $\alpha$.
An automorphism $A$ of the $U_0$-representations will respect each of the
summands $\mathbb{C}^{k_i}_{\beta_i}$ and we see that the centralizer of $U_0$ in $\GL(V)$ is isomorphic to $\prod
\GL(k_i,\mathbb{C})$. But then the centralizer of $U$ in $\GL(V)$ is isomorphic
to the centralizer of the image of a single generator of $G$ in $\prod
\GL(k_i,\mathbb{C})$ which is connected as well, as one can
observe by computing the centralizer of a Jordan block as upper triangular
Toeplitz matrices.
\end{proof} 

\begin{thm}\label{Thm:SWclass}
Let $M$ be an integer GKM manifold with GKM graph $(\Gamma,\alpha)$. Then the image of the total equivariant Stiefel-Whitney class of $M$ in $H^*_T(\Gamma;\mathbb{Z}_2)\oplus B^\ast(\Gamma,2)$ is the total equivariant Stiefel-Whitney class of $(\Gamma;\alpha)$.
\end{thm}

\begin{proof}
The statement about the image in $H^*_T(\Gamma;\mathbb{Z}_2)$ was proved in
\cite[Proposition 3.5]{MR4088417}, by combining naturality of the Stiefel-Whitney classes with the fact that for any fixed point $q$, one may choose an invariant complex structure on $T_qM$, which allows to write the equivariant Stiefel-Whitney classes, restricted to $q$, as the $\mod 2$ reduction of the equivariant Chern classes. It remains to consider an invariant $2$-sphere $S$. We have an equivariant splitting $TM|_S=NS\oplus TS$. Now let $\alpha\in H^2(BT;\mathbb{Z})$ denote the weight of $S$ with an arbitrarily chosen sign and for a fixed point $q\in S^T$ let $\beta_1,\ldots,\beta_{n-1}\in H^2(BT;\mathbb{Z})$ the weights of $NS_q$ where again the sign is chosen arbitrarily. Then we may choose signed weights $\gamma_1,\ldots,\gamma_{n-1}\in H^2(BT;\mathbb{Z})$ at the other fixed point $r$ such that $\beta_i\equiv \gamma_i\mod \alpha$ for $i=1,\ldots,n-1$. There is an equivariant complex line bundle $L_i$ over $S$ with weights $\beta_i,\gamma_i$ over $q,r$, hence by Lemma \ref{lem:uniqueness} we have $NS\cong L_1\oplus\ldots\oplus L_{n-1}$. Furthermore we choose a $T$-invariant complex structure on $S$ such that the weights of $TS$ over $q,r$ become $\alpha,-\alpha$. Hence the total equivariant Stiefel-Whitney class of $TM|_S$ is the $\mod 2$ reduction of the total equivariant Chern class $c$ of $TS\oplus L_1\oplus\ldots\oplus L_{n-1}$. In the GKM description \[H^*_T(S;\mathbb{Z})\rightarrow H^*_T(\{q\};\mathbb{Z})\oplus H^*_T(\{r\};\mathbb{Z})\cong H^*(BT;\mathbb{Z})^2\]
the class $c$ restricts to $\left((1+\alpha)\prod_{i=1}^{n-1}(1+\beta_i),(1-\alpha)\prod_{i=1}^{n-1}(1+\gamma_i)\right)$. Hence, using the terminology from Lemma \ref{L:EquivariantCohmologySalpha}, the image in $B^\ast(\Gamma,2)$ is given by
\[r\left(\frac{(1+\alpha)\prod_{i=1}^{n-1}(1+\beta_i)-(1-\alpha)\prod_{i=1}^{n-1}(1+\gamma_i)}{\alpha}\right).\]
The denominator in the expression above agrees with the one in the definition of the equivariant Stiefel-Whitney class of $(\Gamma;\alpha)$ up to a multiple of $2\alpha$, hence the corresponding elements in $B^\ast(\Gamma,2)$ agree.
\end{proof}
 \begin{rem}\label{rem:GKZ6Korrektur}
In the setting of the Remark \ref{rem:naturality} with $p=2$, it was claimed in
\cite{MR4088417} that the top row of the commutative diagram
from that remark
 maps the equivariant Stiefel--Whitney classes onto one another. This is true and
 follows directly from the results of the last and this section: the vertical
 maps are injective (Theorem \ref{T: InjectionInZp}), the bottom horizontal map
 sends the total equivariant Stiefel--Whitney class of
 $(\Gamma_M,\alpha_M)$ to that of $(\Gamma_N,\alpha_N)$ (Definition
 \ref{defn:eqSW}) and the fact that these are the images of the total
 Stiefel--Whitney classes of $M$ and $N$ (Theorem \ref{Thm:SWclass}). The
 argument given in \cite{MR4088417} was similar and used the same
 commutative diagram but left out the $B^\ast(\Gamma,2)$ summands in the second row. This
 argument however was incomplete as the vertical maps are in general not
 injective without considering the $B^\ast(\Gamma,2)$ summands as is shown by Proposition
 \ref{prop:injectiveiff}.
\end{rem}
 
\begin{cor}\label{cor:nonrealizability}
Let $(\Gamma,\alpha)$ be a GKM-graph of an integer GKM manifold. Then the total equivariant Stiefel--Whitney class lies in the image of the map \[\Psi\colon H^*_T(\Gamma;\mathbb{Z})\rightarrow H^*_T(\Gamma;\mathbb{Z}_2)\oplus B^\ast(\Gamma,2).\]
\end{cor}

\begin{proof}
If $M$ has torsion-free cohomology, then the map $H_T^*(M;\mathbb{Z})\rightarrow H_T^*(M;\mathbb{Z}_p)$ is surjective. The claim now follows from the commutativity of the square in Theorem \ref{T: InjectionInZp}.
\end{proof} 

It is well known that the only obstruction for a spin structure on the frame
bundle of a manifold is given by the second Stiefel--Whitney class. We will show
here the corresponding result for equivariant spin structures. Before we do
that, we would like to settle some necessary definitions. 

Let $X$ be a topological space with a homotopy type of a CW complex. Consider
an oriented (real) vector bundle $E \to X$ endowed with a euclidean bundle metric. Denote
by $P_{\mathrm{SO}}(E)$ the bundle of oriented orthonormal frames of $E \to M$. Now assume
that $X$ is a $G$-CW complex (see \cite[Definition 2.1]{MR358761}) for a connected, compact Lie group $G$ and $E \to M$ is
equivariant vector bundle, such that $G$ acts by isometries on $E$. Then
$P_{\mathrm{SO}}(E)$ has a canonical induced action such that $P_{\mathrm{SO}}(E) \to M$ is equivariant.

\begin{defn}\label{D: equivariant Spin structures}
Let $X$ be a $G$-space and $E \to X$ an $G$-equivariant, real, oriented
euclidean vector bundle of rank $r$ over $X$. An
\emph{equivariant spin structure} on $E \to X$ is a pair $(P_{\mathrm{Spin}}(E),
\varphi)$ such that
\begin{enumerate}[label=(\alph*)]
	\item $P_{\mathrm{Spin}}(E) \to X$ is a $G$-equivariant
		$\mathrm{Spin}(r)$-principal bundle, i.e., the actions of $G$ and
		$\mathrm{Spin}(r)$ commute and the projection to the base is equivariant,
	\item $\varphi \colon P_{\mathrm{Spin}}(E) \to P_{\mathrm{SO}}(E)$ is a
		$G$-equivariant double
		covering,
	\item $\varphi$ is $\mathrm{Spin}(r)$-equivariant in the sense that
		\[
			\varphi(p \cdot s) = \varphi(p) \cdot \rho(s), 
		\]
		where $p \in P_{\mathrm{Spin}}(E)$, $s \in \mathrm{Spin}(r)$ and
		$\rho$ the double covering of $\mathrm{SO}(n)$.
		
\end{enumerate}

\end{defn}

\begin{thm}\label{T: Equivariant spin structures}
Let $X$ be a $G$-CW complex and $E \to X$ be as in the definition above.
Then, $E \to X$ admits an equivariant spin structure if and only if the
equivariant second Stiefel--Whitney class $w_{2}^{G}(E)$ vanishes.

In particular for a GKM manifold $M$ with acting torus $T$ and invariant
Riemannian metric satisfying Condition
\eqref{eq:standingassumption}, the tangent bundle of $M$ admits an equivariant spin structure if and only if 
\begin{enumerate}[label=(\alph*)]
	\item For all $v \in V(\Gamma)$ we have $\sum_{i=1}^{n} \alpha_{i} \equiv 0
		\mod 2$, where $\alpha_{i}$ are the weights of the isotropy representation
		at $v$ and
	\item for all $e \in E(\Gamma,2)$ 
	\[
		\frac{1}{\tilde{\alpha}(e)} \left(\sum_{l\in E_{i(e)}} \tilde{\alpha}(l) - \sum_{l\in E_{t(e)}} \tilde{\alpha}(l)\right)
		\equiv 0 \mod 2,
	\]
	where $\tilde{\alpha}(-)$ are lifts of the labels in $E_{i(e)}$ and $E_{t(e)}$ to $H^\ast(BT;\ZZ)$ such that $\tilde{\alpha}(l)\equiv \tilde{\alpha}(\nabla_e l)\mod \alpha(e)$ for all $l\in E_{i(e)}$.
	 \end{enumerate}
Moreover $M$ admits a (non-equivariant) spin structure if and
only if (same notation as above)
\begin{enumerate}[label=(\alph*)]
	\item  $\sum_{i=1}^{n}\alpha_{i}$ are
		identical $\mod 2$ for all $v \in V(\Gamma)$ and
	\item for all $e \in E(\Gamma,2)$ 
	\[
		\frac{1}{\tilde{\alpha}(e)} \left(\sum_{l\in E_{i(e)}} \tilde{\alpha}(l) - \sum_{l\in E_{t(e)}} \tilde{\alpha}(l)\right)
		\equiv 0 \mod 2.
	\]
\end{enumerate}
\end{thm}

\begin{proof}
Assume first that $w_{2}^{G}(E)=0$. Then the principal bundle
$(P_{\mathrm{SO}}(E))_{G} = EG \times_{G} P_{\mathrm{SO}}(E) \to X_{G}$ admits a spin
structure say $\widetilde P \to (P_{\mathrm{SO}}(E))_{G}$ which is a double
cover. The frame bundle
$P_{\mathrm{SO}}(E)$ has the homotopy type of a $G$-CW complex, therefore we infer
from \cite[Theorem A]{MR711050} that there is a $G$-equivariant double covering
$P \to P_{\mathrm{SO}}(E)$ such that $P_{G} = EG \times_{G} P \cong \widetilde
P$. Since $P \to P_{\mathrm{SO}}(E)$ is the pullback of $P_{G}$ by the
fiber inclusion $P_{\mathrm{SO}}(E) \to (P_{\mathrm{SO}}(E))_{G}$ it follows that $P$ is the pullback of $\widetilde
P$. Thus the concatenation of the maps $P \to P_{\mathrm{SO}}(E) \to X$ is a
$\mathrm{Spin}(r)$-principal bundle, which is the pullback of $\widetilde P \to
X_{G}$ by the fiber inclusion. From the fact that the $G$-action commutes with the orthogonal action on $P_{\mathrm{SO}}(E)$ we infer that the $G$-action on $P$ commutes with the $\mathrm{Spin}(r)$-action: for fixed $g\in G$ and $p\in P$ the maps ${\mathrm{Spin}}(r)\to P$ given by $h\mapsto g\cdot (h\cdot p)$ and $h\mapsto h\cdot (g\cdot p)$ are lifts of the same map to $P_{\mathrm{SO}}(E)$. Hence $P \to X$ is an
equivariant spin structure for $E \to X$.

Conversely, if $P_{\mathrm{Spin}}(E) \to P_{\mathrm{SO}}(E)$ is a $G$-equivariant
spin structure, then $EG \times_{G} P_{\mathrm{Spin}}(E) \to EG \times_{G}
P_{\mathrm{SO}}(E)$  is a spin structure of $EG \times_{G} E$, hence $0 =
w_{2}(EG \times_{G} E) = w_{2}^{G}(E)$.

For the second claim, note first that every $G$-manifold is a $G$-CW complex
(cf. \cite[Proposition 4.4]{MR290354}) and that the $T$ on $M$ induces a canonical action on
$TM \to M$ by taking differentials. This turns $TM \to M$ into a $T$-equivariant
bundle. Now use the combinatorial description of $w_{2}^{T}(M)$ from Theorem \ref{Thm:SWclass}.

Finally, the last claim follows from the fact that the
existence of a spin structure on $M$ is equivalent to $w_{2}(M)$ being zero.
From the naturality property of the Stiefel-Whitney classes the latter condition
holds if and only if $w_{2}^{T}$ lies in the kernel of the map
$H^{2}_{T}(M;\ZZ_{2}) \to H^{2}(M; \ZZ_{2})$ induced by the fiber
inclusion from $M$ into the Borel model. But clearly this is equivalent to the
$H^{2}_{T}(\Gamma; \ZZ_{2})$-part of
$w_{2}^{T}(M)$ fulfilling condition (a) and the $B^2(\Gamma,2)$-part vanishing as in
condition (b).
\end{proof}

\section{Non-realizability via Stiefel--Whitney classes}

\begin{thm}\label{thm:nonrealizable}
In dimension $2n\geq 8$ there is an $n$-valent effective $T^2$-GKM graph $(\Gamma,\alpha)$ such that $H^*_T(\Gamma;\mathbb{Z})$ is free over $H^*(BT;\mathbb{Z})$ and $H^*(\Gamma;\mathbb{Q})=H_T^*(\Gamma;\mathbb{Q})/H^+(BT;\mathbb{Q})\cdot H^*_T(\Gamma;\mathbb{Q})$ satisfies Poincaré duality with fundamental class in degree $2n$ but which is not realizable by an integer GKM manifold.
\end{thm}

\begin{proof}
We prove the statement in dimension $8$ by giving a concrete example. Higher dimensional examples can be constructed by taking the product with single edge graphs with generic label (i.e., GKM graphs of a suitable $S^2$). Let $x,y$ denote a basis of $H^2(BT;\mathbb{Z})$. Consider the GKM graph

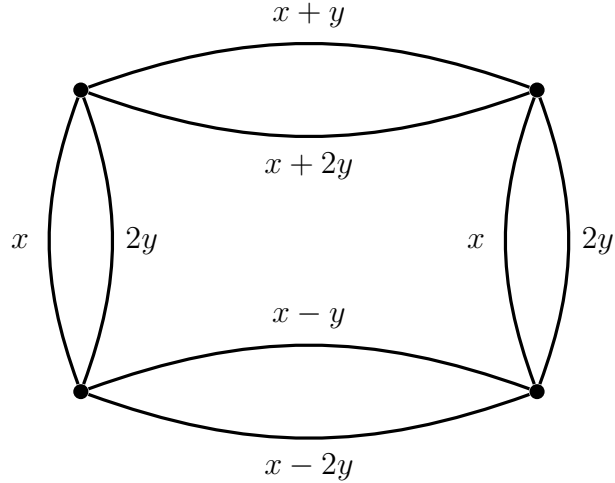
\begin{center}
		\begin{tikzpicture}

			\node (a) at (0,0)[circle,fill,inner sep=2pt] {};
			\node (b) at (6,0)[circle,fill,inner sep=2pt]{};
			\node (c) at (0,4)[circle,fill,inner sep=2pt]{};	
			\node (d) at (6,4)[circle,fill,inner sep=2pt]{};			
			
			\node at (-0.8,2) {$x$};
			\node at (+0.8,2) {$2y$};
			\node at (6-0.8,2) {$x$};
			\node at (6+0.8,2) {$2y$};
			\node at (3,5) {$x+y$};
			\node at (3,3) {$x+2y$};
			\node at (3,1) {$x-y$};
			\node at (3,-1) {$x-2y$};

			\draw [very thick](a) to[in=160, out=20] (b);
			\draw [very thick](a) to[in=200, out=-20] (b);
			\draw [very thick](c) to[in=160, out=20] (d);
			\draw [very thick](c) to[in=200, out=-20] (d);
			\draw [very thick](a) to[in=250, out=110] (c);
			\draw [very thick](a) to[in=290, out=70] (c);
			\draw [very thick](b) to[in=250, out=110] (d);
			\draw [very thick](b) to[in=290, out=70] (d);
			
		\end{tikzpicture}
	\end{center}
	The fact that $H_T^*(\Gamma;\mathbb{Z})$ is free over $H^*(BT;\mathbb{Z})$ and that $H^*_T(\Gamma;\mathbb{Z})$ satisfy Poincaré duality can be checked by computing explicit generators for the above $H^*(BT;\mathbb{Z})$ modules. These turn out to be given by the elements
	\[a_{1} = \begin{pmatrix}
	1\\1\\1\\1
	\end{pmatrix},\quad
	a_{2} = \begin{pmatrix}
	x^2-3xy+2y^2\\
	x^2+3xy+2y^2\\
	0\\ 0
	\end{pmatrix},\quad
	a_{3} = \begin{pmatrix}
	0\\2xy\\2xy\\0
	\end{pmatrix},\quad
	a_{4} = \begin{pmatrix}
	2x^3y-6x^2y^2+4xy^3\\0\\0\\0
	\end{pmatrix}	
	\]
	of $H^*(BT;\mathbb{Z})^{4}$, where the components correspond to the vertices starting with the lower right one and proceeding counterclockwise. The following lines of code carry out this computation in \emph{Macaulay2} 
\begin{center}
\begin{BVerbatim}[fontsize=\scriptsize]
R=ZZ[x,y] 
C=cokernel(diagonalMatrix({x,2*y,x+y,x+2*y,x,2*y,x-y,x-2*y}))
f=map(C,R^4,{{1,-1,0,0},{1,-1,0,0},{0,1,-1,0},{0,1,-1,0},{0,0,1,-1},{0,0,1,-1},{-1,0,0,1},{-1,0,0,1}})
kernel(f)
\end{BVerbatim}
\end{center}

In the code above $R=H^*(BT;\mathbb{Z})$ and $C$ is the direct sum of all
$R/(\alpha(e))$ where $e$ runs over all (non-oriented) edges. The order of the
summands is such that the first two correspond to the double edge on the right
from which point on we proceed counterclockwise while always prioritizing the
left hand resp.\ top edge. The map $f\colon R^4\rightarrow C$ is defined such
that its component corresponding to an edge $e$ is given by taking the
difference of the values at $i(e)$ and $t(e)$ where all edges are oriented
counterclockwise.

Having computed these generators, their linear independence over $R$ can be
computed directly via a determinant argument. To check Poincaré duality we compute
\[a_2a_3=-a_4+2xya_2.\]
	
	 It remains to show that $(\Gamma,\alpha)$ is not geometrically realizable.	The equivariant Stiefel--Whitney class of $(\Gamma,\alpha)$ has a nontrivial $B^\ast(\Gamma,2)$ component in degree $2$. To compute this we consider the left hand vertical edge $e$ of label $2y$ and compute the element $f_e$ as defined in Lemma \ref{lem:defSW}. The sign choices depicted above are admissible in the sense of the definition. Then the degree $2$ part computes as
	\begin{align*}
	r\left(\frac{(x-2y)+(x-y)+x+2y-((x+y)+(x+2y)+x+2y)}{2y}\right)=r\left(\frac{-5y}{y}\right)=1
	\end{align*}
	We argue that the image of $\Psi\colon H_T^*(\Gamma;\mathbb{Z})\rightarrow H^*_T(\Gamma;\mathbb{Z}_2)\oplus B^\ast(\Gamma,2)$ has trivial $B^2(\Gamma,2)$ component, which then implies nonrealizability of $(\Gamma,\alpha)$ by Corollary  \ref{cor:nonrealizability}. Indeed if $f\in H_T^2(\Gamma;\mathbb{Z})$ then $f_{i(e)}\equiv f_{t(e)}\mod x$ and $f_{i(e)}\equiv f_{t(e)}\mod 2y$ for both edges $e\in E(\Gamma,2)$. This implies $f_{i(e)}=f_{t(e)}$ and hence \[r\left(\frac{f_{i(e)}-f_{t(e)}}{2y}\right)=0.\]
\end{proof}

\begin{rem}\label{rem:3valentrealizable}
One should compare this example with the main result of \cite{2210.01856v1} by which a $3$-valent  $T^2$-GKM graph $(\Gamma,\alpha)$ is realizable by an integer GKM manifold if and only if $H^*_T(\Gamma;\mathbb{Z})$ is free over $H^*(BT;\mathbb{Z})$ and $H^*(\Gamma;\mathbb{Q})$ satisfies Poincaré duality with fundamental class in degree $6$. In other words, Theorem \ref{thm:nonrealizable} shows that the necessary and sufficient criteria known for $T^2$-actions in dimensions $\leq 6$ no longer suffice in dimensions $\geq 8$ as the criterion from Corollary \ref{cor:nonrealizability} becomes a nontrivial obstruction. 

We point out that it is shown in \cite[Corollary 2.28]{2210.01856v1} that a
$3$-valent $T^2$-GKM graph such that $H^*_T(\Gamma;\mathbb{Q})$ satisfies Poincaré
duality with fundamental class in degree $6$ is automatically orientable (see
Definition \ref{defn:orientablegkmgraph}). This is why in the realization
theorem of \cite{2210.01856v1} orientability is not listed as a separate condition. In higher dimensions we do not know of any general implication between orientability and Poincaré duality of the equivariant graph cohomology; however, the counterexample in Theorem \ref{thm:nonrealizable} is orientable as well as one can check directly by hand.
\end{rem}

Remark \ref{rem:3valentrealizable} together with Corollary \ref{cor:nonrealizability} imply the purely combinatorical observation that for a $3$-valent graph satisfying the conditions listed in the remark, the total equivariant Stiefel--Whitney class lies in the image of
$H_T^*(\Gamma;\mathbb{Z})\rightarrow H_T^*(\Gamma;\mathbb{Z}_2)\oplus B^\ast(\Gamma,2)$. The remainder of this section is dedicated to giving a direct combinatorial proof of this fact with slightly reduced assumptions.

\begin{prop}\label{prop:SW3valent}
Let $(\Gamma,\alpha)$ be a 3-valent orientable $T^2$-GKM graph. Then the total equivariant Stiefel--Whitney class lies in the image of $H^*_T(\Gamma;\mathbb{Z})\rightarrow H^*_T(\Gamma;\mathbb{Z}_2)\oplus B^\ast(\Gamma,2)$.
\end{prop}

To prove this we first prove the existence of combinatorial Thom classes of connection paths in $3$-valent orientable GKM graphs.

\begin{defn}
Given a connection path $c=e_1,\ldots,e_l$ in a GKM graph as well as a vertex $v$, we call an edge $e$ at $v$ \emph{normal to $c$} if there exists $j$ such that $t(e_{j-1})=i(e_j)=v$ and $e$ is normal to the edge path $e_{j-1},e_j$.
\end{defn}
Note that a vertex may have more than one edge normal to $c$, even if $(\Gamma,\alpha)$ is $3$-valent and orientable.
\begin{ex}
Consider a $2$-valent $2n$-gon with $2n$ vertices and alternating labels $\alpha_1=(1,0)$ and $\alpha_2=(0,1)$. Now take the product of this with the one edge graph of label $\alpha_3=(1,1)$. Note that
\[\alpha_i\equiv \pm \alpha_j\mod\alpha_k\]
for any choice of distinct $i,j,k$. Hence the congruence condition in the definition of a compatible connection is void and any collection of bijections $\{\nabla_e\colon E_{i(e)}\rightarrow E_{t(e)}~|~e\in E(\Gamma), \nabla_e(e)=\overline{e}\}$ defines a compatible connection. Now e.g.\ if one changes the standard product connection at a single edge $e$ belonging to one of the two copies of the $2n$-gon, there will be only a single connection path going through $e$ and which crosses it twice.
\end{ex}

Let $(\Gamma,\alpha)$ be a 3-valent orientable $T^2$-GKM graph with compatible connection $\nabla$ and $c=e_1,\ldots,e_l$ a connection path, i.e.\ a closed edge path satisfying $\nabla_{e_i} e_{i-1}=e_{i+1}$, where we set $e_0=e_l$ and $e_{l+1}=e_1$. We define a function $\beta$ on the set of \emph{oriented} edges of $\Gamma$, as follows: let $f_j$ denote the oriented edge at $i(e_j)$ (i.e., $i(f_j)=i(e_j)$) which is normal to the edge path $e_{j-1},e_j$. We lift the label $\alpha(f_1)$ arbitrarily to $\beta(f_1)\in \ZZ^2$. We transport this chosen sign along $c$ in the sense that, for $j=2,\ldots,l$, we let $\beta(f_j)\in \ZZ^2$ be an element that reduces to $\alpha(f_j)$ modulo $\pm$, with the sign chosen such that
\[
\beta(f_j)\equiv \beta(f_{j-1})\mod \alpha(e_{j-1}).
\]
Note that as connection paths are uniquely determined by two successive edges, every oriented edge can be normal to $c$ only at most once. Thus, $\beta$ is well-defined, but note that with an edge $e$, also the opposite edge $\bar{e}$ might be normal to $c$, and the signs of $\beta(e)$ and $\beta(\bar{e})$ are not directly related. We can extend the function $\beta$ by zero on all edges on which it is not yet defined.

\begin{prop}\label{prop:thomembedded}
Let $(\Gamma,\alpha)$ be a 3-valent orientable $T^2$-GKM graph. For any vertex $v$ we put
\[
\Th(c)_v:= \sum_{e\in E_v} \beta(e).
\]
This gives a well-defined cohomology class $\Th(c)\in H^2_T(\Gamma;\ZZ)$, which we call the \emph{Thom class of $c$}. It is uniquely defined up to a global sign.
\end{prop}

\begin{proof}
We first claim that $\beta(f_1)\equiv \beta(f_l) \mod \alpha(e_l)$, which is the only congruence for normal edges along $c$ that is not automatically satisfied by construction.

In order to show this, we make use of the orientability of $\Gamma$. We lift the labels $\alpha(e_1)$ and $\alpha(e_2)$ to elements $\beta_1$ and $\beta_2\in \ZZ^2$. We choose this notation instead of the more usual $\tilde{\alpha}(e_i)$ as the connection path might traverse any edge twice, even in the same direction. Then, for $i=2,\ldots,l-1$ we lift $\alpha(e_i)$ to $\beta_i\in \ZZ^2$ in such a way that
\[
\beta_{i+1}\equiv -\beta_{i-1}\mod {\alpha}(e_{i}).
\]

The sign choices of $\beta_i$ and $\beta(f_i)$ were arranged such that $\beta_{i+1}\equiv \epsilon_i \beta_{i-1}\mod{\alpha}(e_i)$ holds with $\epsilon_i=-1$ for $i=2,\ldots,l-1$ but the value of $\epsilon_1,\epsilon_l$ is not yet determined. We note that $\det(\beta_{i-1},\beta_i)$ and $\det(\beta_i,\beta_{i+1})$ share the same sign if and only if $\epsilon_i=-1$. As the determinant expression arrives at its original value when moving around the edge path once, we infer that $\epsilon_1=\epsilon_l$.

By the sign choices for $\beta(f_i)$, it follows that $\eta(e_i)=1$ for $i=2,\ldots,l-1$ and hence by orientability of $\Gamma$ we have $\eta(e_1)=\eta(e_l)$. If $\eta(e_l)=-1$, then $1=\epsilon_l=\epsilon_1$. Then $\eta(e_1)=-1$ implies that $\beta(f_1)\equiv \beta(f_l)\mod {\alpha}(e_l)$. If $\eta(e_l)=1$ we arrive at the same conclusion as this forces $\epsilon_1=\epsilon_l=-1$ and thus $\eta_1=1$ implies $\beta(f_1)\equiv \beta(f_l)\mod {\alpha}(e_l)$.

Now we need to show that $\Th(c)$, as defined in the statement of the proposition, satisfies all congruence relations to be contained in the equivariant graph cohomology. For an edge $e$ that is not part of the connection path $c$, the values $\Th(c)_{i(e)}$ and $\Th(c)_{t(e)}$ are either zero or lifts of $\alpha(e)$ to $\ZZ^2$, which makes the congruence relation along $e$ trivially satisfied.

For an edge $e$ that is part of the connection path, $\Th(c)_{i(e)}$ and $\Th(c)_{t(e)}$ consist each of at most three summands, corresponding to parts of $c$ that traverse $i(e)$ respectively $t(e)$. Potential parts of $c$ with $e$ respectively $\bar{e}$ as normal edge are irrelevant for the congruence relation, as they contribute only lifts of $\alpha(e)$. The at most two other parts contribute summands of the form $\beta(f_j)$ and $\beta(f_{j+1})$ to $\Th(c)_{i(e)}$ respectively $\Th(c)_{t(e)}$, which satisfy the congruence relation by construction of $\beta$ and the argument in the first part of the proof.
\end{proof}

As a corollary of our construction of Thom classes, we obtain that the Stiefel--Whitney classes of a $3$-valent orientable GKM graph lie in the image of $H^*_T(\Gamma;\ZZ)\to H^*_T(\Gamma;\ZZ_2)\oplus B^\ast(\Gamma,2)$:
\begin{proof}[Proof of Proposition \ref{prop:SW3valent}]
We claim that the second equivariant Stiefel--Whitney class of $\Gamma$ is equal to the image of the sum $\sum_c \Th(c)$, where $c$ runs through all connection paths of $\Gamma$ (up to starting vertex and orientation). In a $3$-valent graph every edge at a vertex $v$ arises exactly once as the normal edge to an edge path through $v$. The sum $\sum_c \Th(c)$ therefore evaluates, at any vertex $v$, as the sum of the labels of the adjacent edges, with certain signs. As with $\ZZ_2$ coefficients the signs do not matter, this shows that the $H^2_T(\Gamma;\ZZ_2)$-components of the second Stiefel--Whitney class and of the image of $\sum_c \Th(c)$ coincide. As for the $B^2(\Gamma,2)$-component, consider any edge $e\in E(\Gamma,2)$. Denote by $e_1,e_2$ the other two edges at $i(e)$, as well as $f_i:=\nabla_e e_i$. Then the $B^2(\Gamma,2)$-component of the second Stiefel--Whitney class at $e$ is
\[
r\left(\frac{\tilde{\alpha}(e_1)+\tilde{\alpha}(e_2)-\tilde{\alpha}(f_1)-\tilde{\alpha}(f_2)}{\tilde{\alpha}(e)}\right),
\]
where the signs are chosen in any way such that that congruences $\tilde{\alpha}(e_i)\equiv \tilde{\alpha}(f_i) \mod \alpha(e)$ hold. But these congruences are fulfilled by the chosen signs in the Thom classes of the connection paths following the edge paths $e_1,e,f_1$ and $e_2,e,f_2$, by construction of the function $\beta$.

Concerning the fourth equivariant Stiefel--Whitney class of $\Gamma$, we claim
that it is equal to the image of the sum $\sum_e \Th(e)$ of all Thom classes of
edges of $\Gamma$. Recall from \cite[Definition 2.15]{2210.01856v1} that the
Thom class $\Th(e)\in H^4_T(\Gamma;\ZZ)$ of an edge $e$ is defined as follows:
with notation as before, i.e., $e_1,e_2$ and $f_1,f_2$ the other edges at $i(e)$
and $t(e)$ respectively, with $f_i=\nabla_e e_i$ and signs chosen such that
$\tilde{\alpha}(f_i)\equiv \tilde{\alpha}(e_i)\mod \alpha(e)$, then
\[
\Th(e)_v=\begin{cases} \tilde{\alpha}(e_1) \tilde{\alpha}(e_2) & v = i(e) \\ \tilde{\alpha}(f_1)\tilde{\alpha}(f_2) & v = t(e). \end{cases}
\] 
Firstly this shows that, denoting $E_v=\{e_1,e_2,e_3\}$ for a vertex $v$,
\[
\left(\sum_e \Th(e)\right)_v = \tilde{\alpha}(e_1)\tilde{\alpha}(e_2) + \tilde{\alpha}(e_1)\tilde{\alpha}(e_3) + \tilde{\alpha}(e_2)\tilde{\alpha}(e_3)
\]
which coincides, after passing to $\ZZ_2$ coefficients, with the $H^4_T(\Gamma;\ZZ_2)$-component of the fourth Stiefel--Whitney class of $\Gamma$. Now fix any edge $e\in E(\Gamma,2)$. Then the $B^4(\Gamma,2)$-component of the fourth Stiefel--Whitney class is by definition given by
\begin{align*}
r\left(\frac{\tilde{\alpha}(e_1)\tilde{\alpha}(e_2) +\tilde{\alpha}(e)\tilde{\alpha}(e_1) + \tilde{\alpha}(e)\tilde{\alpha}(e_2)- \tilde{\alpha}(f_1)\tilde{\alpha}(f_2) -\tilde{\alpha}(\bar{e})\tilde{\alpha}(f_1) -\tilde{\alpha}(\bar{e})\tilde{\alpha}(f_2)}{\tilde{\alpha}(e)}\right).
\end{align*}
which coincides with the $e$-component in $B^4(\Gamma,2)$ of the sum of Thom classes.

As for the sixth equivariant Stiefel--Whitney class, we claim that it is equal to
the image of the sum $\sum_v \Th(v)$ of the Thom classes of all vertices of
$\Gamma$. Recall from \cite{MR1823050}, cf.\ \cite[Definition 2.13]{2210.01856v1} that the Thom class $\Th(v)\in H^6_T(\Gamma;\ZZ)$ of a vertex $v$ is well-defined up to sign, and given by
\[
\Th(v)_u = \begin{cases} \prod_{e\in E_v} \tilde{\alpha}(e) & u=v \\ 0 & u \neq v,  \end{cases}
\] 
with signs chosen arbitrarily. 

This shows directly that the $H^6_T(\Gamma;\ZZ_2)$-part of $\sum_v \Th(v)$ and of the sixth Stiefel--Whitney class coincide. One computes that the $B^6(\Gamma,2)$-part of both classes vanishes.
\end{proof} 

\printbibliography
\end{document}